\documentclass[12pt,draftcls,onecolumn]{IEEEtran}
\usepackage{amsmath,amssymb,array,graphicx,palatino,verbatim}

\usepackage{latexsym,url}
\usepackage{amsmath,amssymb,amsthm}
\usepackage{algorithm,algorithmic}
\usepackage{color,array,graphicx,verbatim,xtab}
\usepackage{cite} 
\usepackage{setspace}

\usepackage
  [breaklinks,bookmarks,bookmarksnumbered,bookmarksopen,bookmarksopenlevel=2]
  {hyperref}

\newcommand{\reals}{{\mathbb{R}}}

\newcommand{\vect}[1]{\boldsymbol{#1}}
\newcommand{\ie}{{\it i.e.}}
\newcommand{\eg}{{\it e.g.}}

\newtheorem{theorem}{Theorem}
\newtheorem{lemma}[theorem]{Lemma}
\newtheorem{definition}{Definition}

\newtheorem{remark}{Remark}
\newtheorem{corollary}{Corollary}

\usepackage{color}
\newcommand{\mycomment}[3]%
{%
\marginpar{%
  \hfil%
  \tiny{\textcolor{#2}{{\bf\textsc{#1}}}}%
  \hfil%
}%
\footnote{\textcolor{#2}{{\bf\textsc{#1}:}~~#3}}
}
\newcommand{\kakhbod}[1]%
{\mycomment{kakhbod}{blue}{#1}}
\newcommand{\jckoo}[1]%
{\mycomment{jckoo}{red}{#1}}

\begin{document}

\title{Market Mechanisms with Non-Price-Taking Agents}
\author{
{\Large {Ali Kakhbod}} \\
Department of Electrical Engineering and Computer Science \\
University of Michigan, Ann Arbor, MI, USA.\\
Email: {\texttt{akakhbod@umich.edu}}}

\date{May 2011}
\maketitle

\begin{abstract}
\textbf{The paper develops a  decentralized resource allocation mechanism for allocating  divisible goods with capacity constraints to  non-price-taking agents with general concave utilities. The proposed mechanism is always budget balanced, individually rational, and it converges to an optimal solution of the corresponding centralized problem. Such a mechanism is very useful in a network with general topology and no auctioneer  where the competitive agents/users want different type of services.}
\end{abstract}

\begin{section}{Introduction}
\label{1}
\begin{subsection}{Motivation and Challenges}
Most of today's market networks support the delivery of a variety of goods to their agents/ individuals.
One of the main challenges in market networks is the design of resource allocation strategies
which guarantee the delivery of different goods/services and maximize some performance
criterion (e.g. the network's utility to its agents/individuals). The challenge in determining
such resource allocation strategies comes from the fact that each agent's utility is its own
private information. The network (network manager) is unaware of the agents' utilities. If
information were centralized, the resource allocation problem could be formulated and solved as a
mathematical programming problem or as a dynamic programming problem. Since information is
not centralized such formulations are not possible. The challenge is to determine a decentralized
message exchange process among the agents and an allocation rule (based on the outcome of
the message exchange process) that eventually lead to a resource allocation that is optimal for
the corresponding centralized problem.\\

The above considerations have guided most of the research on resource allocation in market
networks. Basically, there are two microeconomic approaches have been used for the development of efficient decentralized resource allocation schemes in market networks: resource-oriented and
price-oriented \cite{hurwicz1973}.\\

In the resource-oriented approach, each individual computes the marginal values for his current
resources, and communicates them to the rest of the agents. The allocation is then changed so
that agents with an above average marginal utility receive more of this resource and agents with
a below average marginal utility receive less. This approach has been used in \cite{kurose1989} to develop
decentralized algorithms for optimally allocating a single resource to a set of interconnected
computing individuals. In the price-oriented approach, an initial allocation of resources is made
and an arbitrary set of systemwide initial resource prices is chosen. Prices then are iteratively
changed to accommodate the demands for resources until the total demand for a resource
exactly equals the total amount available. Most of the results on decentralized resource allocation
currently available in the literature are based on the price-oriented approach \cite{cocchi1993,jordan1995,49,wang1997,murphy19941,murphy19942,parris1992,kelly1994,kelly1998,gupta1997,coucoubetis1997,mackie-mason1995,low1993,deveciana1998,thomas1997,thomas2002}. In this
paper, since we follow the price-oriented approach to resource allocation, we critically review
the results reported in \cite{cocchi1993,jordan1995,49,wang1997,murphy19941,murphy19942,parris1992,kelly1994,kelly1998,gupta1997,coucoubetis1997,mackie-mason1995,low1993,deveciana1998,thomas1997,thomas2002} so that we can point out to the contributions of our work.\\

The work currently available on decentralized resource allocation by price-oriented methods has addressed, either by analysis \cite{wang1997,49,murphy19941,kelly1994,kelly1998,gupta1997,coucoubetis1997,mackie-mason1995,low1993,deveciana1998,thomas1997}, or simulation and
analysis \cite{cocchi1993,murphy19941,murphy19942,parris1992,gupta1997}. A significant number of publications have dealt with single good
\cite{wang1997,murphy19941,murphy19942,parris1992,deveciana1998}, or with the allocation of a single resource per connection \cite{jordan1995,49,wang1997,murphy19941,murphy19942,parris1992,kelly1994,kelly1998,gupta1997,coucoubetis1997,mackie-mason1995}.\\

In several papers \cite{wang1997,49,murphy19941,low1993,deveciana1998,thomas1997} the following general philosophy to resource allocation by price-oriented methods has been adopted: (1) formulate a centralized constrained optimization problem where the objective is the maximization of a social welfare function and the constraints are imposed by the availability of network resources; and (2) use
pricing methods to devise a decentralized scheme that realizes the solution of the centralized
problem and satisfies the informational constraints imposed by the network. The existence of
a solution to the centralized problem is shown, and market methods are used to structure and
develop the solution. The existence of a set of prices that induce agents to request the optimal allocation is established and, in some cases \cite{wang1997,gupta1997,low1993,deveciana1998,thomas1997} an iterative
scheme for adjusting the prices based on the agent's requests is described. However, none of the papers specify a mechanism to force the successive prices to converge to the optimal set of prices.
\\

Furthermore, in all of the aforementioned papers it is assumed that individuals/agents act as if
their behavior has no effect on the equilibrium prices reached by the market/message allocation
process. In other words, there is an auctioneer/network manager who updates the prices according
to the excess demand while the agents update their demands based on the prices (i.e. the network
periodically adjusts the prices based on the monitored agent request for resources).\\

In this paper, we follow the price-oriented approach and the philosophy presented in the
previous paragraph, to address the resource allocation problem in market networks. Our formulation
of the resource allocation problem captures the issues and considerations discussed in the
first paragraph of this section. In particular, we consider a connection-oriented network that can
potentially offer multiple services/goods to individuals. agents' preferences are summarized by
means of their utility functions, and each agent is allowed to request more than one type of
service/good. No assumption is made on the functional form of the utility functions, although
some mild regularity conditions are imposed (see Section II). We assume that the relation between
the availability of goods and resource allocation is given (see discussion in Section II), and we
incorporate it as a constraint into a static optimization problem (the centralized problem). The
objective of the optimization problem is to determine the amount and, required resources for
each type of service/good to maximize the sum of the agents' utilities. We prove the existence
of a solution of the optimization problem, and describe a competitive market economy that
implements the solution and satisfies the informational constraints imposed by the nature of
the decentralized resource allocation problem. In particular, in contrast to other papers that
follow the philosophy of price-oriented method we assume individuals/agents act as if their
behavior has a direct effect on the equilibrium prices reached by the market allocation process,
we justify it in Section I-B, due to this relaxation agents in market networks called non-price
taking. This relaxation is justified in fully decentralized networks/systems in which there is no
auctioneer (equivalently the designer is purely wealth redistributionary, see Section II) to set the
price and thus, due to the nature of the market problems [29], individuals should converge, via a
decentralized message exchange process, to the same price in a distributed manner. A philosophy
similar to ours with different objectives for the public good problem has also been adopted in \cite{shroti2}.

\end{subsection}

\begin{subsection}{Contribution of the paper}
\textit{The contributions of this paper are:}
\begin{enumerate}
	\item \textit{The formulation of a general optimization problem with non-price taking agents that
request multiple types of divisible goods with capacity constraint.}
  \item \textit{The discovery of a novel decentralized message exchange process/mechanism for allocating
divisible goods with capacity constraints to non-price-taking agents with general concave
utility functions which requires minimal coordination overhead and possesses the following
desirable properties:}
\newcounter{Lcount}  
\begin{list}{(P\arabic{Lcount})}{\usecounter{Lcount}%
\addtolength{\leftmargin}{-\labelwidth}%
\settowidth{\labelwidth}{(P\arabic{Lcount})}%
\addtolength{\leftmargin}{\labelwidth}}
	\item \textit{The mechanism is always budget-balanced at every allocation.}
	\item \textit{The mechanism converges to an allocation (that we call it efficient stationary profile) corresponding to an optimal solution of the centralized problem.}
	\item \textit{Every efficient stationary profile resulted from the mechanism is weakly preferable
than the initial endowment of each agent.}
\end{list}
\end{enumerate}

To the best of our knowledge, none of the decentralized resource allocation mechanisms
proposed so far for allocating divisible goods with capacity constraints possesses simultaneously
properties (P1-P3) with general concave utility functions and non-price-taking agents. 
\end{subsection}
\begin{subsection}{Organization of the paper}
The rest of the paper is organized as follows.
In section \ref{model} we formulate the centralized  problem with non-price-taking agents. In section \ref{gamef} we describe the elements of the  mechanism we propose for achieving a solution of the  centralized problem. In section \ref{pro} we analyze the properties of the proposed mechanism.   We conclude in section \ref{con}.\\


\end{subsection}
\end{section}
\begin{section}{Model and Objectives}
\label{model}
\begin{subsection}{General Model}
In this problem, we have $m$ non-price-taking agents and the set of agents is
$\mathcal{A} = \{1,2,\ldots,m\}$.  Suppose there are $n$ goods which
are each infinitely divisible; let $\mathcal{L} =
\{l_1,l_2,\ldots,l_n\}$ be the set of goods.  There is only a limited
amount of each good; the maximum amount of $l_j$ available is denoted
by $c_{l_j}$, for $j=1,2,\ldots,n$, and is always nonnegative. Consider a specific agent $i$.  Let $\mathcal{L}_i =
\{l_{i,1},l_{i,2},\ldots,l_{i,|\mathcal{L}_i|}\} \subseteq
\mathcal{L}$ be the subset of goods which may be requested by agent
$i$.  For each agent, the subset $\mathcal{L}_i$ is known and fixed in
advance.  The amount of goods actually demanded by the agent is given
by the demand vector $\vect{x}_i =
(x_{i,1},x_{i,2},\ldots,x_{i,|\mathcal{L}_i|})$, where $x_{i,k}$ is
the amount of good $l_{i,k}$ requested by $i$, for $k =
1,2,\ldots,|\mathcal{L}_i|$.  For a given demand $\vect{x}_i$, the
utility to agent $i$ is the function $U_i : \reals^{|\mathcal{L}_i|}
\to \reals$, $U_i \in \mathcal{U}$, which is differentiable, concave, and satisfies
$U_i(\vect{0}) = 0$ and $U_i(\vect{z}) = -\infty$ if any entry of
$\vect{z}$ is negative.  Also, call $\vect{x} =
(\vect{x}_1,\ldots,\vect{x}_m)$ to be the overall demand of all
agents.  Furthermore, let $\mathcal{A}_l \subseteq \mathcal{A}$ be the
set of agents who may request good $l$; \ie, $\mathcal{A}_l = \{i : l
\in \mathcal{L}_i\}$ \footnote{We assume $|A_l|>2, \forall l\in\mathcal{L}$}.  
Let us denote by $-i$ to be the set of agents who are not $i$; \ie,
$-i = \mathcal{A} \setminus \{i\}$.  Then we can call $\vect{x}_{-i} =
\{\vect{x}_j : j \ne i\}$ and $\mathcal{A}_l^{-i} = \mathcal{A}_l
\setminus \{i\}$. 

Suppose that we have a designer (\eg, network manager) who wants to
design a mechanism to maximize the social welfare $\sum_{i=1}^m [
U_i(\vect{x}_i) - t_i(\vect{x}_i) ]$.  Here, $t_i :
\reals^{|\mathcal{L}_i|} \to \reals$ consist of taxation policies on
agents $i$, $i=1,\ldots,m$, and is to be designed.  We consider a
scenario where the role of the designer is purely wealth
redistributionary; there is no net tax collected, so $\sum_{i=1}^m
t_i(\vect{x}_i) = 0$.  Then, we can write the tax-explicit social
welfare maximization problem, that we call it \textbf{Max1}, as the following:
\begin{eqnarray}
\begin{array}{ll}
\mbox{maximize} & \displaystyle \sum_{i=1}^m [ U_i(\vect{x}_i)
- t_i(\vect{x}_i) ] \\
\mbox{subject to}
& \displaystyle \sum_{i=1}^m t_i(\vect{x}_i) = 0 \\
& \displaystyle \sum_{i: l \in \mathcal{L}_i} x_{i,l} \leq c_l, \; \;
\forall l \in \mathcal{L}
\end{array}
\label{eq:tax_explicit_market_problem}
\end{eqnarray}

where the optimization variables are $\vect{x}_i \in \reals^{|\mathcal{L}_i|}$ for all $i \in \mathcal{A}$.  The resulting
taxes charged to (or accrued by) the agents are denoted by the vector
$\vect{t} = (t_1,t_2,\ldots,t_m)$, with $t_i$ being shorthand for
$t_i(\vect{x}_i)$.


\end{subsection}
\begin{subsection}{The decentralized problem  and the Objective}
We consider the model of the previous subsection
with the following assumptions on its information structure.

\begin{list}{(A\arabic{Lcount})}{\usecounter{Lcount}%
\addtolength{\leftmargin}{-\labelwidth}%
\settowidth{\labelwidth}{(A\arabic{Lcount})}%
\addtolength{\leftmargin}{\labelwidth}}
\item Each agent only knows his own utility; this utility is his own private information.
\item There is no auctioneer to set/adjust the price (price vector) per unit of service/good(s).
\item Agents talk to each other in a broadcast setting, i.e., each agent hears every other agent's message. Therefore, agent $i$ does not need to be aware of $\mathcal{A}_l^{-i}$, for every $l, l\in \mathcal{L}_i$.
\end{list}

From the above description it is clear that the information in the network is decentralized.

Under the above  assumptions the objective is to determine a mechanism which has the following properties:

\newcounter{Xcount}  
\begin{list}{(P\arabic{Xcount})}{\usecounter{Xcount}%
\addtolength{\leftmargin}{-\labelwidth}%
\settowidth{\labelwidth}{(P\arabic{Xcount})}%
\addtolength{\leftmargin}{\labelwidth}}
    \item  The mechanism is individually rational  (the agents voluntarily participate in the mechanism).
    \item The mechanism is  balanced budget at every allocation. 
    \item  The mechanism converges to an optimal solution of the corresponding centralized problem \textbf{Max1}. 
\end{list}

\begin{remark}
Before proceeding with the specification of our mechanism, we further comment
on the appropriateness of the non-price taking behavior of individuals. The non-price taking
behavior is inspired by the following problem/environment:
\begin{quote}
``Consider a fully informationally decentralized system that only consists of a community
of \textbf{competitive individuals} and a set of goods that the individuals are
interested in obtaining them (or a subset of them). Under the assumption that the satisfaction
(utility) of each individual-which is a function of his demand profile (amount
of goods allocated to him) and his tax payment-is its private information, what is the
optimal demand profile for each individuals that maximize the social welfare?''
\end{quote}
As it is observed from the above environment, individuals (should) act as if their behavior has
a direct effect on the equilibrium prices reached by the market allocation process. Further, since there is no auctioneer who could be able to collect the taxes, the budget should also be balanced  because otherwise there might be money left unallocated at the end of the allocation
process.
\end{remark}
\end{subsection}
\end{section}

\begin{section}{A mechanism for allocating divisible goods}
\label{gamef}
\begin{subsection}{The components of the mechanism}
A mechanism is described by $(\mathcal{M},f)$, where $\mathcal{M}=\times_{i=1}^m\mathcal{M}_i$ is the message space, specifying for each agent $i, i\in \mathcal{A}$, the set of messages $\mathcal{M}_i$ that agent $i$ uses to communicate to other agents, and $f$ is an outcome function that describes the actions that are taken for every $\vect{m}:=(\vect{m}_1,\vect{m}_2,\cdots,\vect{m}_m)\in \mathcal{M}$. The mechanism $(\mathcal{M},f)$ is common knowledge among the agents.

For the decentralized resource allocation problem formulated in section \ref{model}  we propose a mechanism the components of which we describe below. 

\noindent\textbf{Message space}: The message space for agent $i$,
$i\in \mathcal{A}$, is given by
$\mathcal{M}_i \subset \mathbb{R}_+^{2|\mathcal{L}_i|}$. Specifically a message of
agent $i$ is of the form $$ \vect{m}_i=(x_{i,l_1},x_{i,l_2},\cdots,x_{i,l_{|\mathcal{L}_i|}},p_{i,l_1},p_{i,l_2},\cdots,p_{i,l_{|\mathcal{L}_i|}})=(\vect{x}_i,\vect{p}_i)$$  where $0 \leq x_{i,l} \leq c_l$ and $0\leq p_{i,l}\leq M, \ \forall i \in \mathcal{A}, \ l \in \mathcal{L}_i$, where $0<M<\infty$. The component $x_{i,l}$ denotes the amount of good $l, l\in \mathcal{L},$ that agent $i$ requests. The component $p_{i,l}, l \in \mathcal{L}_i$, denotes the price that agent $i$ is willing to pay for the unit of good $l$. 

\noindent\textbf{Outcome Function}: At any time slot $n, n=1,2,\cdots $, the outcome function  $f$ is given by
$$
f: \mathcal{M}=\mathcal{M}_1 \times \mathcal{M}_2 \times \cdots \times
\mathcal{M}_m \rightarrow \mathbb{R}_+^m \times \mathbb{R}^{|\mathcal{L}_1|} \times \mathbb{R}^{|\mathcal{L}_2|}
\cdots  \times \mathbb{R}^{|\mathcal{L}_m|}
$$
 and is defined as follows. For any $$\vect{m}^{(n)}:=(\vect{m}_1^{(n)},\vect{m}_2^{(n)},\cdots, \vect{m}_m^{(n)}) \in \mathcal{M},$$
 \begin{eqnarray}
 f(\vect{m}^{(n)})=f(\vect{m}_1^{(n)},\vect{m}_2^{(n)}, \cdots, \vect{m}_m^{(n)})=({\vect{x}}_1^{(n)},{\vect{x}}_2^{(n)},\cdots,{\vect{x}}_m^{(n)},\vect{t}_1(\vect{m}^{(n)}), \vect{t}_2(\vect{m}^{(n)}), \cdots , \vect{t}_m(\vect{m}^{(n)}))\nonumber
 \end{eqnarray}
 where at any time $n, n = 1, 2, ...,$ (superscript $(n)$ denotes time slot $n$) $\vect{x}_i^{(n)}=(x_{i,l_1}^{(n)},x_{i,l_2}^{(n)},\cdots,x_{i,l_{|\mathcal{L}_i|}}^{(n)})$ 
 and $x_{i,l}^{(n)}, i \in \mathcal{A},$ is the amount of good $l$ allocated to agent $i$, and $\vect{t}_i(\vect{m}^{(n)})=(t_{i,l_1}^{(n)},t_{i,l_2}^{(n)},\cdots,t_{i,l_{|\mathcal{L}_i|}}^{(n)})$, $i\in \mathcal{A}$ ($t_{i,l}$ is the tax (subsidy) agent $i$ pays (receives) for good $l, l \in \mathcal{L}_i$). We proceed now to specify $t_{i,l}^{(n)}, l\in \mathcal{L}_i$.
 
 For every $i \in \mathcal{A}, l\in \mathcal{L}_i, n=1,2,3,...$

\begin{eqnarray}
\label{tax}
\label{2}
t_{i,l}^{(n)}&=&w_{-i,l}^{(n)}x_{i,l}^{(n)}+\frac{1}{\kappa^{(n)}}|p_{i,l}^{(n)}-\bar{p}_l^{(n)}|^2\nonumber \\&&-w_{-i,l}^{(n)}\left(p_{i,l}^{(n)}-p_{-i,l}^{(n)}\right)\left(\frac{x_{i,l}^{(n)}+\sum_{j\neq i, j\in \mathcal{A}_l}{x_{j,l}^{(n)}-c_l}}{\gamma}\right) \nonumber \\
&&+\frac{1\{x_{i,l}^{(n)}>0\}1\{\mathcal{E}_{-i,l}^{(n)}+x_{i,l}^{(n)}>0\}}{1-1\{x_{i,l}^{(n)}>0\}1\{\mathcal{E}_{-i,l}^\textit{l}+x_{i,l}^{(n)}>0\}}+\phi_{i,l}^{(n)}\mathcal{I}\{|\mathcal{A}_l|>2\}
\end{eqnarray}
 where $\gamma$ is fixed and sufficiently large, 
 \begin{eqnarray}
 &&0<\theta^{(n+1)}<\theta^{(n)},  \; \; \lim_{n\rightarrow \infty} \theta^{(n)}=0, \nonumber \\
 &&\lim_{n\rightarrow \infty} \kappa^{(n)}=\infty, \; \;  \kappa^{(n)}:=\sum_{k=1}^n\theta^{(k)}, \ \mbox{for example $\theta^{(n)}=\frac{1}{n}, n=1,2,3,...$}
 \end{eqnarray}
 \[ 1\{A\} = \left\{ \begin{array}{ll}
         1-\epsilon  & \mbox{if $A$ holds};\\
        0 & \mbox{otherwise},\end{array} \right. \] 
where $\epsilon$ is bigger than zero and  sufficiently small,\footnote{Therefore, when A and B (both) hold, then $\frac{1\{A\}1\{B\}}{1-1\{A\}1\{B\}}\approx\frac{1}{0^+}$ is well defined and it becomes a large number.} and
 \[ \mathcal{I}\{A\} = \left\{ \begin{array}{ll}
         1 & \mbox{if $A$ holds};\\
        0 & \mbox{otherwise},\end{array} \right. \]
\begin{eqnarray}
&&w_{i,l}^{(n)}=\frac{\sum_{k=1}^n\theta^{(k)}p_{i,l}^{(k)}}{\kappa^{(n)}}, \\
&& \bar{p}_l^{(n)}=\frac{\sum_{i\in \mathcal{A}_l}p_{i,l}^{(n)}}{|\mathcal{A}_l|},  \\
&&{p}_{-i,l}^{(n)}=\frac{\sum_{\substack{j \in \mathcal{A}_l \\ j \neq i}}p_{j,l}^{(n)}}{|\mathcal{A}_l|-1}, \\
&&  w_{l}^{(n)}=\frac{\sum_{k=0}^{n-1}\theta^{(k+1)}\bar{p}_{l}^{(k)}}{\kappa^{(n)}}, \\
&&w_{-i,l}^{(n)}= \frac{\sum_{j \in \mathcal{A}_l,j\neq i}w_{i,l}^{(n)}}{|\mathcal{A}_l|-1}, \\
&&\mathcal{E}_{-i,l}^{(n)}=\left\{\sum_{\substack{j\in \mathcal{A}_l, j \neq i}}x_j^{(n)}\right\}-c_l, \\
&&\mathcal{E}_{i,l}^{(n)}=(|\mathcal{A}_l|-1)x_i^{(n)}-c_l,
\end{eqnarray}
\begin{eqnarray}
\label{bterm}
\phi_{i,l}^{(n)}&=&-\frac{1}{\kappa^{(n)}|\mathcal{A}_l|^2}\left[\left(|\mathcal{A}_l|-1\right)\sum_{\substack{j\in \mathcal{A}_l \\ j \neq i}}(p_j^{(n)})^2+\left(\sum_{\substack{j\in \mathcal{A}_l \\ j \neq i}}p_j^{(n)}\right)^2\right]\nonumber \\
&&\!\!+\frac{1}{\kappa^{(n)}|\mathcal{A}_l|^2}\left[\frac{2\left(|\mathcal{A}_l|-1\right)\sum_{\substack{j \in \mathcal{A}_l\\ j \neq i}}\sum_{\substack{k \in \mathcal{A}_l \\ k\neq i, j}}p_{j,l}^{(n)}p_{k,l}^{(n)}}{|\mathcal{A}_l|-2}\right]\nonumber \\
&&\!\!-\frac{\sum_{\substack{j \in \mathcal{A}_l\\ j \neq i}}\sum_{\substack{k \in \mathcal{A}_l \\ k\neq i, j}}x_{j,l}^{(n)}w_{k,l}^{(n)}}{\left(|\mathcal{A}_l|-2\right)\left(|\mathcal{A}_l|-1\right)}+\frac{\sum_{\substack{j \in \mathcal{A}_l\\ j \neq i}}\sum_{\substack{k \in \mathcal{A}_l \\ k\neq i, j}}p_{j,l}^{(n)}x_{j,l}^{(n)}w_{k,l}^{(n)}}{\gamma \left(|\mathcal{A}_l|-2\right)\left(|\mathcal{A}_l|-1\right)}\nonumber \\
&&\!\!-\frac{\sum_{\substack{j \in \mathcal{A}_l\\ j \neq i}}\sum_{\substack{k \in \mathcal{A}_l \\ k\neq i, j}}\sum_{\substack{r \in \mathcal{A}_l \\ r\neq i, j,k}}p_{k,l}^{(n)}x_{j,l}^{(n)}w_{r,l}^{(n)}}{(|\mathcal{A}_l|-1)^2 \gamma (|\mathcal{A}_l|-3)}-\frac{w_{-i,l}^{(n)}p_{-i,l}^{(n)}\mathcal{E}_{-i,l}^{(n)}}{\gamma}\nonumber \\
&&\!\!+\frac{\sum_{\substack{j \in \mathcal{A}_l\\ j \neq i}}\sum_{\substack{k \in \mathcal{A}_l \\ k\neq i, j}}p_{k,l}^{(n)}\mathcal{E}_{j,l}^{(n)}w_{j,l}^{(n)}}{(|\mathcal{A}_l|-1)^2 \gamma (|\mathcal{A}_l|-2)}-\frac{\sum_{\substack{j \in \mathcal{A}_l\\ j \neq i}}\sum_{\substack{k \in \mathcal{A}_l \\ k\neq i, j}}p_{k,l}^{(n)}x_{j,l}^{(n)}w_{k,l}^{(n)}}{(|\mathcal{A}_l|-1)^2 \gamma (|\mathcal{A}_l|-2)} \nonumber \\
&&\!\!+\frac{\sum_{\substack{j \in \mathcal{A}_l\\ j \neq i}}\sum_{\substack{k \in \mathcal{A}_l \\ k\neq i, j}}\sum_{\substack{r \in \mathcal{A}_l \\ r\neq i, j,k}}p_{r,l}^{(n)}\mathcal{E}_{j,l}^{(n)}w_{k,l}^{(n)}}{(|\mathcal{A}_l|-1)^2 \gamma (|\mathcal{A}_l|-3)}.
\end{eqnarray}
Next we specify additional subsidies $Q^i$ that agent $i, i\in \mathcal{A}$, may receive. For that matter we consider all goods $l \in \mathcal{L}$ such that $|\mathcal{A}_l| = 3$. For each good $l$, with $|\mathcal{A}_l| = 3$ we define the 
\begin{eqnarray}
\label{5005}
Q_{\{l:|\mathcal{A}_l|=3\}}^{(n)}&=&-w_{-i,l}^{(n)}x_{i,l}^{(n)}-\frac{1}{\kappa^{(n)}}|p_{i,l}^{(n)}-\bar{p}_l^{(n)}|^2\nonumber \\
&&+w_{-i,l}^{(n)}\left(p_{i,l}^{(n)}-p_{-i,l}^{(n)}\right)\left(\frac{x_{i,l}^{(n)}+\sum_{j\neq i, j\in \mathcal{A}_l}{x_{j,l}^{(n)}-c_l}}{\gamma}\right) \nonumber \\
&\stackrel{(a)}{=}&o(1)-w_{-i,l}^{(n)}x_{i,l}^{(n)}-\frac{1}{\kappa^{(n)}}|p_{i,l}^{(n)}-\bar{p}_l^{(n)}|^2,
\end{eqnarray}
where $(a)$ follows since $\gamma$ is chosen sufficiently large.\\
Furthermore for each good $l\in \mathcal{L}$ where $|\mathcal{A}_l|=3$ an agent $k_l\in\mathcal{A}_l$ is chosen, randomly, to assign the subsidy $Q_{\{l:|\mathcal{A}_l|=3\}}^{(n)}$. Let $l_1,l_2,\cdots,l_r$ be the set of goods such that $|\mathcal{A}_{l_i}|=3, i=1,2,3,\cdots,r,$ be the corresponding agents that receive $Q_{\{l:|\mathcal{A}_l|=3\}}^{(n)}$.\\
Based on the above, the tax (subsidy) paid (received) by agent $j,j\in \mathcal{A}$,  is the following. If $j\neq k_{l_1}, k_{l_2},\cdots k_{l_r}$ then
\begin{eqnarray}
t_j^{(n)}=\sum_{l\in\mathcal{L}_j}t_{j,l}^{{n}},
\end{eqnarray}
If $j=k_{l_i}, i=1,2,\cdots,r,$ then
\begin{eqnarray}
t_{k_{l_i}}^{(n)}=\sum_{l\in \mathcal{L}_{k_{l_i}}}t_{k_{l_i},l}^{(n)}+Q_{\{l_i:|\mathcal{A}_{l_i}|=3\}}^{(n)},
\end{eqnarray}
where $Q_{\{l:|\mathcal{A}_l|=3\}}^{(n)}$ is defined by \eqref{5005}.\\

Note that $Q_{\{l_i:|\mathcal{A}_{l_i}|=3\}}^{(n)}$ is not controlled by agent $k_{l_i}$ , that is, $Q_{\{l_i:|\mathcal{A}_{l_i}|=3\}}^{(n)}$ does not depend on agent $k_{l_i}$ 's message. Thus, the presence (or absence) of $Q_{\{l_i:|\mathcal{A}_{l_i}|=3\}}^{(n)}$ does not influence the strategic behavior of agent $k_{l_i}$ . We have assumed here that the agents $k_{l_1}, k_{l_2}, \cdots, k_{l_r}$ , are distinct. Expressions similar to the above hold when the agents $k_{l_1}, k_{l_2}, \cdots, k_{l_r}$ are not distinct.

\begin{remark}
For each good $l \in \mathcal{L}$ with $|\mathcal{A}_l| = 3$, we could equally divide the subsidy $Q_{\{l:|\mathcal{A}_{l}|=3\}}^{(n)}$ among all agents not in $\mathcal{A}_l$ instead of randomly choosing one agent $k\in \mathcal{A}_l$. Any other division
of the subsidy $\mathcal{A}_l$ among agents not in $\mathcal{A}_l$ would also work.
\end{remark}
\subsection{The dynamic of the mechanism}
The dynamic of the mechanism is as follows. Note that $(\vect{x}^{(1)},\vect{p}^{(1)})$ is given, arbitrary, and feasible.
\begin{itemize}
	\item At time $n\geq 1$: every agent $i, i \in \mathcal{A}$ solves and reports
\begin{eqnarray}
\label{1}	
(\vect{x}_i^{(n+1)},\vect{p}_i^{(n+1)}):=\arg \max_{\vect{x}_i,\vect{p}_i} \left\{U_i(\vect{x}_i^{(n)})-\sum_{l \in \mathcal{L}_i}t_{i,l}^{(n)}\right\}
\end{eqnarray}
where 
\begin{eqnarray}
\label{2}
t_{i,l}^{(n)}&=&w_{-i,l}^{(n)}x_{i,l}+\frac{1}{\kappa^{(n)}}|p_{i,l}-\bar{p}_l^{(n)}|^2\nonumber\\
&&-w_{-i,l}^{(n)}\left(p_{i,l}-p_{-i,l}^{(n)}\right)\left(\frac{x_{i,l}+\sum_{j\neq i, j\in \mathcal{A}_l}{x_{j,l}^{(n)}-c_l}}{\gamma}\right)\nonumber \\
&&+\phi_{i,l}^{(n)}
\end{eqnarray}
 and $\phi_{i,l}^{(n)}$ is defined as \eqref{bterm}.
\end{itemize}
Before proceeding further we define some notations and concepts.
\begin{definition}
A taxation scheme $\vect{t} = (t_1,\ldots,t_m)$ is
\emph{budget-balanced} if the sum of taxes is zero; \ie, $\sum_{i=1}^m
t_i = 0$. 
\end{definition}
In other words, budget-balance means that
the sum of taxes paid by some of the agents is equal to the sum of the money
(subsidies) received by the rest of the agents participating in the allocation process. Budget balance implies that there is no money left unallocated at the end of the allocation.
\begin{definition}[Stationary profile (SP)]
A message profile $\vect{m}^*=(\vect{m}_1^*,\vect{m}_2^*,\cdots,\vect{m}_m^*), \vect{m}^* \in \mathcal{M}$, is called stationary profile  if for any $i, i \in \mathcal{A},$
\begin{eqnarray}
U_i(\vect{x}_i^*)-\vect{t}_i(\vect{m}_i^*,\vect{m}_{-i}^*)\geq U_i(\vect{x}_i)-\vect{t}_i(\vect{m}_i,\vect{m}_{-i}^*) \quad \forall \ \vect{m}_i\in \mathcal{M}_i.  
\end{eqnarray}
\end{definition}
In the other words, a stationary profile is a message profile that unilateral deviation is not profitable for any  agents\footnote{It is important to note that the nature of the stationary profile is similar to the concept of  Nash equilibrium. But, in  Nash equilibrium agents' utilities are common knowledge among the agents where as here  agents' utilities are their private information.}.  
\begin{definition}[Individual rationality]
A message profile $\vect{m}=(\vect{m}_1,\vect{m}_2,\cdots,\vect{m}_m), \vect{m} \in \mathcal{M}$, is weakly  preferred by all agents to the initial allocation $(\vect{0},0)$ if 
\begin{eqnarray}
U_i(\vect{x}_i)-\vect{t}_i(\vect{m})\geq U_i(\vect{0})-0=0, \quad \quad \forall \ i\in \mathcal{A}.  
\end{eqnarray}
A mechanism is called Individually rational if the above property is satisfied at any stationary profile, that is, agents  are incentivized to  participate voluntarily in the mechanism.
\end{definition}
 In other words, Individual rationality  asserts that at any stationary profile the utility of each user is at least as much
as its utility before participating in the mechanism.
\end{subsection}
\begin{subsection}{Interpretation of the mechanism}
It observes that the designer of the mechanism can not alter the
agents' utility functions, that is $U_i, i \in \mathcal{A},$ is agent $i'$s private information. Therefore,  the only way  to  achieve the optimal solution of the  corresponding centralized problem is  through the use of appropriate tax functions/incentive. For each good, the tax functions of our mechanism consists of three components $\Upsilon_1^{(n)}$, $\Upsilon_2^{(n)}$ and $\Upsilon_3^{(n)}$, at each time slot $n$. We specify
and interpret these components in the following.
\\
At each time slot $n, n=1,2,\cdots,$ equation \eqref{tax} can be decomposed as follows, 
$$
t_{i,l}^{(n)}=\Upsilon_1^{(n)}+\Upsilon_2^{(n)}+\Upsilon_3^{(n)} \\
$$
where
\begin{eqnarray}
\Upsilon_1^{(n)}&:=&w_{-i,l}^{(n)}x_{i,l}^{(n)} \nonumber \\
\Upsilon_2^{(n)}&:=& \frac{1}{\kappa^{(n)}}|p_{i,l}^{(n)}-\bar{p}_l^{(n)}|^2\nonumber\\
&-&\!\!\!\!\!\!\!w_{-i,l}^{(n)}\left(p_{i,l}^{(n)}-p_{-i,l}^{(n)}\right)\left(\frac{x_{i,l}^{(n)}+\sum_{j\neq i, j\in \mathcal{A}_l}{x_{j,l}^{(n)}-c_l}}{\gamma}\right) \nonumber \\  &+&\!\!\!\!\!\!\!\frac{1\{x_{i,l}^{(n)}>0\}1\{\mathcal{E}_{-i,l}^{(n)}+x_{i,l}^{(n)}>0\}}{1-1\{x_{i,l}^{(n)}>0\}1\{\mathcal{E}_{-i,l}^\textit{l}+x_{i,l}^{(n)}>0\}}
 \nonumber \\
\Upsilon_3^{(n)}&:=& \phi_{i,l}^{(n)} \nonumber
\end{eqnarray}
\begin{itemize}
    \item $\Upsilon_1^{(n)}$ specifies the  amount  agent $i$ has to pay for good $l, l\in \mathcal{L}_i$. It is important to note that the price per unit of a good  that  agent $i$ pays is determined by the message/proposal of the other agents requesting the same good. Thus,  agent $i$ does not control the price it pays or receives.
    \item $\Upsilon_2^{(n)}$ provides the following incentives to the agents, $i, i \in \mathcal{A}_l$: (1) To bid/propose the same price per unit of  good $l, l\in \mathcal{L}.$ (2) To collectively request  a total request for the good  does not exceed the available capacity of the good.  The incentive provided to all agents to bid the same price per unit of good $l$ is captured by the
    term $\frac{1}{\kappa^{(n)}}|p_{i,l}^{(n)}-\bar{p}_l^{(n)}|^2$. The incentive provided to all agents to collectively request a total request that does not exceed the available capacity is imposed by the term
    $$
    \frac{1\{x_{i,l}^{(n)}>0\}1\{\mathcal{E}_{-i,l}^{(n)}+x_{i,l}^{(n)}>0\}}{1-1\{x_{i,l}^{(n)}>0\}1\{\mathcal{E}_{-i,l}^\textit{l}+x_{i,l}^{(n)}>0\}},
    $$
Note that an agent is  very  heavily (infinitely)  penalized if it
requests a nonzero amount of a good, and, collectively,  all the agents of the good
request a total  that exceeds the available capacity. A
joint incentive provided to all agents to bid the same price per
unit of good and to utilize the total capacity of the good is
captured by the term
$$
w_{-i,l}^{(n)}\left(p_{i,l}^{(n)}-p_{-i,l}^{(n)}\right)\left(\frac{x_{i,l}^{(n)}+\sum_{j\neq i, j\in \mathcal{A}_l}{x_{j,l}^{(n)}-c_l}}{\gamma}\right)
$$
\item $\Upsilon_3^{(n)}$, the goal of this component is to lead to  a balanced budget.
Further, it is important to note that $\Upsilon_3^{(n)}$ is not controlled by the agent $i$'s message, thus, it does not influence over its  behavior. 
\end{itemize}
\end{subsection}
\end{section}
\begin{section}{Properties of the mechanism} 
\label{pro}
In this section we prove that the mechanism proposed in Section \ref{gamef} has the following properties:  It is individually
rational.  It is budget-balanced at every feasible allocation. It converges to an optimal solution of the centralized problem \textbf{Max1}. \\
We establish the above properties by proceeding as follows. First we show some properties of a stationary profile of the mechanism, Lemma \ref{lemma5}. Then, we show that
agents voluntarily participate in the allocation process. We do this by showing
that the allocations they receive at all efficient stationary profile  of the mechanism is weakly preferred to the $(\vect{0}, 0)$ allocation they receive when they do not participate in the mechanism,  Theorem \ref{ind rational}. Afterwards, we establish that the
mechanism is budget-balanced at all feasible allocations, Lemma \ref{lemma3}.  Finally, we show that the mechanism
converges to an optimal solution of the centralized allocation problem \textbf{Max1}, Theorem \ref{th1}. \\
We present the proofs of the following theorems and lemmas in Appendix.\\
The following lemma presents some key properties of a stationary profile.
\begin{lemma}
\label{lemma5} \label{nash pe0} 
Let $\vect{m}^\ast=(\vect{x}^\ast,\vect{p}^\ast)$
be a stationary profile. Then for every $l \in \mathcal{L}$ and
$i \in \mathcal{A}_l$, we have,
\begin{eqnarray}
\label{nashlagt}
p_{i,l}^{\ast}&=&p_{j,l}^{\ast}=p_{l}^{\ast}, \\
\label{nashlagtt}
w_{i,l}^{\ast}\Big(\frac{\sum_{i, i\in \mathcal{A}_l}x_{i,l}^{\ast}-c_l}{\gamma}\Big)&=&0, \\
\label{nashlagttt} 
\frac{\partial U_i(\vect{x}_i)}{\partial x_{i,l}}\bigg|_{\vect{m}=\vect{m}^*}&=&w_{-i,l}^{\ast}.
\end{eqnarray}
\end{lemma}
In the rest of the paper, we restrict our attention to  the class of  stationary profiles defined as follows.
\begin{definition}
\label{effNE}
Any stationary profile $\vect{m}^*$ of the  mechanism is called efficient stationary profile if for any $i, i\in \mathcal{A}$ and $\l, l\in \mathcal{L},$ the following holds
\begin{eqnarray}
w_{i,l}^{\ast}=w_{j,l}^{\ast}=w_{-i,l}=:w_{l}^{\ast} \quad \ i,j \in \mathcal{A}_l, \ l \in \mathcal{L}.
\end{eqnarray}
\end{definition}
The following theorem  shows that the  mechanism proposed in Section \ref{gamef} is individually rational, i.e., agents voluntarily participate in the mechanism. 
\begin{theorem}
\label{ind rational} 
Every efficient stationary profile, $\vect{m}^*=(\vect{x}^\ast,\vect{t}^\ast)$, of the mechanism 
is weakly preferred by all  agents to the initial allocation
$(\vect{0}, 0)$.
\end{theorem}
By the following theorem, we show that, along the allocation process,  the sum of taxes paid by some of the
agents is always equal to the sum of the money (subsidies) received by the rest of the agents
participating in the allocation process. 
\begin{theorem}
\label{lemma3}
The proposed mechanism is always budget balanced at every feasible allocation. That is, the mechanism is budget-balanced at all allocations. 
\end{theorem}
Finally, we show that the mechanism converges to an allocation corresponding to a solution of the centralized allocation problem \textbf{Max1}.
 \begin{theorem}
 \label{th1}
 The sequence $\{w_{i,l}^{(n)}\}_{n=1}^{\infty}$ tends to the corresponding Lagrangian multipliers of the centralized problem \textbf{Max1} and $\{x_{i,l}^{n}\}_{n=1}^{\infty}$ converges to the optimal solution of \textbf{Max1}. In other words, the proposed mechanism converges to an efficient stationary profile that corresponds to the solution of \textbf{Max1}.
 \end{theorem}
\end{section}
\begin{section}{Conclusion}
\label{con}
In this paper we have proposed a  mechanism for allocating  divisible goods with  capacity constraints  to  non-price-taking agents with general  concave utility which possesses the following desirable properties.
(i) The mechanism is always budget-balanced at every allocation.
(ii) The mechanism converges to an stationary profile (efficient stationary profile) that corresponds to an optimal solution of the corresponding centralized problem.
(iii) Every stationary profile (efficient stationary profile) resulted from the mechanism is weakly preferable than initial endowment of each agent.
\subsection*{Acknowledgments}
 The author gratefully acknowledges stimulating discussions with Demos Teneketzis.
\end{section}



\bibliographystyle{IEEEtran}
\bibliography{references}



\appendix

\begin{proof}[\textbf{Proof of Lemma \ref{lemma5}}]
Consider agent $i \in \mathcal{A}_l$. Since agent $i$ does not control $\phi_{i,l}$ defined by \eqref{bterm}, (i.e. $\phi_{i,l}$ does not depend on $x_{i,l}$ and $p_{i,l}$),
\begin{eqnarray}
\label{333}
\frac{\partial \phi_{i,l}}{\partial
x_{i,l}}=\frac{\partial \phi_{i,l}}{\partial
p_{i,l}}=0.
\end{eqnarray}
Now, Let $\vect{m}^*$ be a stationary profile, then 
\begin{eqnarray}
\label{eq1}
&&\frac{\partial \{U_i(\vect{x}_i)-\sum_{l \in \mathcal{L}_i}t_{i,l}\}}{\partial x_{i,l}}|_{\vect{m}=\vect{m}^*}=0 \\
\label{eq2}
&&\frac{\partial \{U_i(\vect{x}_i)-\sum_{l \in \mathcal{L}_i}t_{i,l}\}}{\partial p_{i,l}}|_{\vect{m}=\vect{m}^*}=0. 
\end{eqnarray}
Equation \eqref{eq2} implies that 
\begin{eqnarray*}
\label{eq3}
2\frac{1- \frac{1}{|\mathcal{A}_l|}}{\kappa^{(SP)}}\left(p_{i,l}^*-\bar{p}_l^*\right)-w_{-i,l}^*\left(\frac{x_i^*+\sum_{j \in \mathcal{A}_l, j\neq i}x_j^*-c_l}{\gamma}\right)=0,
\end{eqnarray*} 
where $\kappa^{(SP)}$ means at the stationary profile $\vect{m}^*$.
Taking a sum over all $i, i \in \mathcal{A}_l,$ and using the fact that $\sum_{i \in \mathcal{A}_l}(p_{i,l}^*-\bar{p}_l^*)=0$, result in
\begin{eqnarray}
\sum_{i \in \mathcal{A}_l}w_{-i,l}^*\left(\frac{x_i^*+\sum_{j \in \mathcal{A}_l, j\neq i}x_j^*-c_l}{\gamma}\right)=0.
\end{eqnarray}
Therefore, for every $i,$
\begin{eqnarray}
\label{eq4}
w_{-i,l}^*\left(\frac{x_i^*+\sum_{j \in \mathcal{A}_l, j\neq i}x_j^*-c_l}{\gamma}\right)=0.
\end{eqnarray}
Substituting \eqref{eq4} in \eqref{eq3} gives 
\begin{eqnarray}
\label{eq5}
p_{i,l}^*=\bar{p}_l^*, \quad \quad \forall \   l \in \mathcal{L},  \  i \in \mathcal{A}_l,
\end{eqnarray}  
and consequently, 
\begin{eqnarray}
\label{eq6}
p_{i,l}^*=p_{j,l}^*=p_{-i,l}^*, \quad \quad \forall\  i, \ j \in \mathcal{A}_l, \  l \in \mathcal{L}.
\end{eqnarray}
Equation \eqref{eq1} together with \eqref{eq6} and form of the tax (Eq. \eqref{tax}) imply that
\begin{eqnarray}
\label{eq7}
\frac{\partial U_i(\vect{x}_i)}{\partial x_{i,l}}|_{\vect{m}=\vect{m}^*}=w_{-i,l}^*.
\end{eqnarray}
\end{proof}

\begin{proof}[\textbf{Proof of Theorem \ref{ind rational}}]
 As it is pointed out in section \ref{model}, for simplicity we assume  for every $i, i\in \mathcal{A},$ $U_i(\vect{0})=0,$ and initial endowment of each agent is zero.
Let $\vect{m}^*=(\vect{x}^\ast,\vect{t}^\ast)$ be a  stationary profile then we have 
\begin{eqnarray}
\label{nash}
\{U_i(\vect{x}_i)-t_i\}|_{\vect{m}=\vect{m}^*}\geq \{U_i(\vect{x}_i)-t_i\}|_{\vect{m}=(\vect{m}_{i},\vect{m}_{-i}^*)} 
\end{eqnarray}

Thus, it is enough to find $\vect{m}_{i}=(\vect{x}_i,\vect{p}_i)=(x_{i,l_1},x_{i,l_2},\cdots,x_{i,l_{|\mathcal{L}_i|}},p_{i,l_1},p_{i,l_2},\cdots,p_{i,l_{|\mathcal{L}_i|}}), l_k \in \mathcal{L}_i, 1 \leq k \leq |\mathcal{L}_i|,$ such that $\{U_i(\vect{x}_i)-t_i\}|_{\vect{m}=(\vect{m}_{i},\vect{m}_{-i}^*)}$ becomes zero. To do so, choose 
\begin{eqnarray}
\label{eq30}
\vect{x}_i=\vect{0}=(0,0,\cdots,0),
\end{eqnarray}
and for any $l, l\in \mathcal{L}_i$, choose $p_{i,l}=\eta_{i,l}$, 

\begin{eqnarray}
\label{eq31}
\eta_{i,l}&=&\frac{\sqrt{\left(w_l^*\frac{\mathcal{E}_{-i,l}^*}{\gamma}\right)^2+4\frac{1}{\kappa^{(SP)}}\left(\frac{|\mathcal{A}_l|-1}{|\mathcal{A}_l|}\right)^2w_l^*x_{-i,l}^*}+2\frac{1}{\kappa^{(SP)}}\left(\frac{|\mathcal{A}_l|-1}{|\mathcal{A}_l|}\right)^2p_l^*+w_l^*\frac{\mathcal{E}_{-i,l}^*}{\gamma}}{2\frac{1}{\kappa^{(SP)}}\left(\frac{|\mathcal{A}_l|-1}{|\mathcal{A}_l|}\right)^2}
\end{eqnarray}
where $\eta_{i,l}$ is the positive root of the following quadratic polynomial with respect to $p_{i,l}$,
\begin{eqnarray*}
\label{eq32}
t_{i,l}|_{\vect{m}=(\vect{m}_{i},\vect{m}_{-i}^*)}&=&\frac{1}{\kappa^{(SP)}}\left(\frac{|\mathcal{A}_l|-1}{|\mathcal{A}_l|}\right)^2(p_{i,l}-p_l^*)^2
-w_{l}^*(p_{i,l}-p_l^*)\left(\frac{\mathcal{E}_{-i,l}^*}{\gamma}\right)+\phi_l^* \nonumber \\
&=&p_{i,l}^2\left[\frac{1}{\kappa^{(SP)}}\left(\frac{|\mathcal{A}_l|-1}{|\mathcal{A}_l|}\right)^2\right]-p_{i,l}\left[2p_l^*\frac{1}{\kappa^{(SP)}}\left(\frac{|\mathcal{A}_l|-1}{|\mathcal{A}_l|}\right)^2+w_l^*\frac{\mathcal{E}_{-i,l}}{\gamma}\right] \nonumber \\
&&+ \left[\frac{1}{\kappa^{(SP)}}\left(\frac{|\mathcal{A}_l|-1}{|\mathcal{A}_l|}\right)^2(p_l^*)^2+w_l^*p_l^*\frac{\mathcal{E}_{-i,l}^*}{\gamma}-w_l^*x_{-i,l}^*\right],
\end{eqnarray*}
where 
\begin{eqnarray}
\label{eq33}
\phi_l^*&=&-w_l^*x_{-i,l}^*,\\
x_{-i,l}^*&=&\frac{\sum_{j \in \mathcal{A}_l,j \neq i}x_j^*}{|\mathcal{A}_l|-1}.
\end{eqnarray} 
Therefore, choosing $\vect{m}_{i}=(\vect{0},\eta_{i,l_1},\eta_{i,l_2},\cdots,\eta_{i,l_{|\mathcal{L}_i|}}), l_k \in \mathcal{L}_i, 1 \leq k \leq |\mathcal{L}_i|,$ implies that
\begin{eqnarray}
\{U_i(\vect{x}_i)-t_i\}|_{\vect{m}=(\vect{m}_{i},\vect{m}_{-i}^*)}=U_i(\vect{0})-\sum_{l, l\in \mathcal{L}_i}t_{i,l}|_{\vect{m}=(\vect{m}_{i},\vect{m}_{-i}^*)}=0.
\end{eqnarray}
Consequently, 
\begin{eqnarray}
\{U_i(\vect{x}_i)-t_i\}|_{\vect{m}=(\vect{m}^*)}\geq 0, \quad \quad \forall i.
\end{eqnarray}
\end{proof}

\begin{proof}[\textbf{Proof of Theorem \ref{lemma3}}]
We want to show that $\sum_{i, i \in \mathcal{A}}t_i=\sum_{l, l \in \mathcal{L}}\left[\sum_{i, i\in \mathcal{A}_l}t_{i,l}\right]=0$. To prove this, we show that $\sum_{i, i\in \mathcal{A}_l}t_{i,l}=0$ for any $l, l \in \mathcal{L}$. 
By a little algebra we can show the following equalities,
\begin{eqnarray}
\label{forty5}
\sum_{i \in \mathcal{A}_l}w_{-i,l}^{(n)}x_{i,l}^{(n)}&=&\sum_{i \in \mathcal{A}_l}\Psi_1(i), \\
\sum_{i \in \mathcal{A}_l}\frac{1}{\kappa^{(n)}}|p_{i,l}^{(n)}-\bar{p}_l^{(n)}|^2&=&\sum_{i \in \mathcal{A}_l}\Psi_2(i),\\ 
\sum_{i \in \mathcal{A}_l}\left[-w_{-i,l}^{(n)}p_{i,l}^{(n)}\frac{x_{i,l}^{(n)}}{\gamma}\right]&=&\sum_{i \in \mathcal{A}_l}\Psi_3(i), 
\end{eqnarray}
\begin{eqnarray}
\sum_{i \in \mathcal{A}_l}\left[-w_{-i,l}^{(n)}p_{i,l}^{(n)}\frac{\mathcal{E}_{-i,l}^{(n)}}{\gamma}\right]&=&\sum_{i \in \mathcal{A}_l}\Psi_4(i), \\
\label{fforty5}
\sum_{i \in \mathcal{A}_l}\left[w_{-i,l}^{(n)}p_{-i,l}^{(n)}\frac{x_{i,l}^{(n)}}{\gamma}\right]&=&\sum_{i \in \mathcal{A}_l}\Psi_5(i), 
\end{eqnarray}
where 
\begin{eqnarray*}
\Psi_1(i)=\frac{\sum_{\substack{j \in \mathcal{A}_l\\ j \neq i}}\sum_{\substack{k \in \mathcal{A}_l \\ k\neq i, j}}x_{j,l}^{(n)}w_{k,l}^{(n)}}{\left(|\mathcal{A}_l|-2\right)\left(|\mathcal{A}_l|-1\right)}, 
\end{eqnarray*}
\begin{eqnarray*}
\Psi_2(i)=\frac{1}{\kappa^{(n)}|\mathcal{A}_l|^2}\Bigg[\left(|\mathcal{A}_l|-1\right)\sum_{\substack{j\in \mathcal{A}_l \\ j \neq i}}(p_j^{(n)})^2+\left(\sum_{\substack{j\in \mathcal{A}_l \\ j \neq i}}p_j^{(n)}\right)^2 -\frac{2\left(|\mathcal{A}_l|-1\right)\sum_{\substack{j \in \mathcal{A}_l\\ j \neq i}}\sum_{\substack{k \in \mathcal{A}_l \\ k\neq i, j}}p_{j,l}^{(n)}p_{k,l}^{(n)}}{|\mathcal{A}_l|-2}\Bigg],
\end{eqnarray*}
\begin{eqnarray*}
\Psi_3(i)=-\frac{\sum_{\substack{j \in \mathcal{A}_l\\ j \neq i}}\sum_{\substack{k \in \mathcal{A}_l \\ k\neq i, j}}p_{j,l}^{(n)}x_{j,l}^{(n)}w_{k,l}^{(n)}}{\gamma \left(|\mathcal{A}_l|-2\right)\left(|\mathcal{A}_l|-1\right)},
\end{eqnarray*}
\begin{eqnarray*}
\Psi_4(i)=-\Bigg[\frac{\sum_{\substack{j \in \mathcal{A}_l\\ j \neq i}}\sum_{\substack{k \in \mathcal{A}_l \\ k\neq i, j}}\sum_{\substack{r \in \mathcal{A}_l \\ r\neq i, j,k}}p_{r,l}^{(n)}\mathcal{E}_{j,l}^{(n)}w_{k,l}^{(n)}}{(|\mathcal{A}_l|-1)^2 \gamma (|\mathcal{A}_l|-3)}
+\frac{\sum_{\substack{j \in \mathcal{A}_l\\ j \neq i}}\sum_{\substack{k \in \mathcal{A}_l \\ k\neq i, j}}p_{k,l}^{(n)}\mathcal{E}_{j,l}^{(n)}w_{j,l}^{(n)}}{(|\mathcal{A}_l|-1)^2 \gamma (|\mathcal{A}_l|-2)}\Bigg],
\end{eqnarray*}
\begin{eqnarray}
\label{ffforty5}
\Psi_5(i)= \Bigg[\frac{\sum_{\substack{j \in \mathcal{A}_l\\ j \neq i}}\sum_{\substack{k \in \mathcal{A}_l \\ k\neq i, j}}\sum_{\substack{r \in \mathcal{A}_l \\ r\neq i, j,k}}p_{k,l}^{(n)}x_{j}^{(n)}w_{r,l}^{(n)}}{(|\mathcal{A}_l|-1)^2 \gamma (|\mathcal{A}_l|-3)}
+\frac{\sum_{\substack{j \in \mathcal{A}_l\\ j \neq i}}\sum_{\substack{k \in \mathcal{A}_l \\ k\neq i, j}}p_{k,l}^{(n)}x_{j}^{(n)}w_{k,l}^{(n)}}{(|\mathcal{A}_l|-1)^2 \gamma (|\mathcal{A}_l|-2)}\Bigg].
\end{eqnarray}
Furthermore, 
\begin{eqnarray}
\label{bbr}
\phi_{i,l}^{(n)}
=-\Bigg\{\Psi_2(i)+\Psi_1(i)+\Psi_3(i)+\Psi_4(i)+\Psi_5(i)
+\frac{w_{-i,l}^{(n)}p_{-i,l}^{(n)}\mathcal{E}_{-i,l}^{(n)}}{\gamma}\Bigg\}.
\end{eqnarray}
Therefore, from Eqs. \eqref{forty5}-\eqref{bbr} we obtain that for any $l, l\in \mathcal{L},$ $\sum_{i, i\in \mathcal{A}_l}t_{i,l}=0$ because,
\begin{eqnarray}
&&\sum_{i, i \in \mathcal{A}_l}\Bigg\{w_{-i,l}^{(n)}x_{i,l}^{(n)}+\frac{1}{\kappa^{(n)}}|p_{i,l}^{(n)}-\bar{p}_l^{(n)}|^2 
-w_{-i,l}^{(n)}\left(p_{i,l}^{(n)}-p_{-i,l}^{(n)}\right)\left(\frac{x_{i,l}^{(n)}+\sum_{j\neq i, j\in \mathcal{A}_l}{x_{j,l}^{(n)}-c_l}}{\gamma}\right)\Bigg\}\nonumber \\&&+\sum_{i, i \in \mathcal{A}_l}\phi_{i,l}^{(n)}=0.
\end{eqnarray}
\end{proof}

\begin{proof}[\textbf{Proof of Theorem \ref{th1}}]
By rearranging \eqref{1} we have the following,
\begin{eqnarray}
U_i(\vect{x}_i^{(n)})-\sum_{l,l\in \mathcal{L}_i}t_{i,l}^{(n)}\Big|_{\vect{x}=\vect{x}^{(n)},\vect{p}_{-i}=\vect{p}_{-i}^{(n)}}
= \Upsilon_i(\vect{p}_i)-\sum_{l,l \in \mathcal{L}_i}\frac{|p_{i,l}-\bar{p}_{l}^{(n)}|^2}{\kappa^{(n)}}
\end{eqnarray}
where 
\begin{eqnarray*}
	\Upsilon_i(\vect{p}_i)= U_i(\vect{x}_i^{(n)})-\sum_{l \in \mathcal{L}_i}\Bigg\{w_{-i,l}^{(n)}x_{i,l}^{(n)}-w_{-i,l}^{(n)}\left(p_{i,l}-p_{-i,l}^{(n)}\right)\left(\frac{x_{i,l}^{(n)}+\sum_{j\neq i, j\in \mathcal{A}_l}{x_{j,l}^{(n)}-c_l}}{\gamma}\right) +\phi_{i,l}^{(n)}\Bigg\}.
\end{eqnarray*}

To prove the convergence, first we state and prove some Lemmas we need to prove the theorem.
\begin{lemma}
\label{lem0}
For any $l, l \in \mathcal{L}$ and $i, i \in \mathcal{A}_l$, there exists at least one subsequence (with the same indices) of $w_{i,l}^{(n)}$ which is convergent. 
\end{lemma}
\begin{proof}
The proof  directly follows by employing  Bolzano-Weierstrass theorem, and the fact that $p_{i,l} \in [0, M]$ for any $i \in \mathcal{A}$ and $l \in \mathcal{L}$.
\end{proof}
\begin{lemma}
\label{lem1}
The following holds for any $t, t\in \mathbb{N}$, $i, i \in \mathcal{A}$ and $l, l \in \mathcal{L}$
\begin{eqnarray}
\kappa^{(t+1)}\Upsilon_i(\vect{p}_i^{(t+1)})-\sum_{l \in \mathcal{L}_i}|p_{i,l}^{(t+1)}-p_l|^2+\sum_{l \in \mathcal{L}_i}|\bar{p}_{l}^{(t)}-p_l|^2
-\sum_{l \in \mathcal{L}_i}|p_{i,l}^{(t+1)}-\bar{p}_l^{(t)}|^2 \geq \kappa^{(t+1)}\Upsilon_i(\vect{p}) 
\end{eqnarray}
for any $\vect{p} \in \Delta$\footnote{$\Delta, \Delta \subset \mathbb{R}^{|L|}$ is a bounded convex set that every bundle of the prices can be chosen from it}.  
\end{lemma}

\begin{proof}
The proof of this Lemma comes directly from the extension of Theorem 1.6\footnote{Theorem 1.6 \cite{lion}: Consider the function $J(v)=J_1(v)+J_2(v)$ where we assume that the functions $J_i(v), i=1,2,$ are continuous, convex, and lower semi-continuous in the weak topology. Further let $$J(v)\rightarrow +\infty  \  \mbox{as} \ ||v||\rightarrow +\infty.$$ We assume that the function $v \rightarrow J_1(v)$ is differentiable, but $J_2$ is not necessarily differentiable. Finally assume that $J$ is strictly convex. Then the unique element of $u$ such that $J(u)=\inf_{v} J(v)$ is characterized by$$J_1'(u)\cdot(v-u)+J_2(v)-J_2(u)\geq 0 \quad \forall \ v.$$}, p. 12 of \cite{lion}, by taking $|\cdot|^2$ as function $J_1(\cdot)$ and $U_i(\cdot)$ as function $J_2(\cdot)$. 
\end{proof}
\begin{lemma}
\label{lem2}
We have the following for any $i, i \in \mathcal{A}$
\begin{eqnarray}
\frac{\sum_{l \in \mathcal{L}_i}\sum_{k=0}^{n-1}|p_{i,l}^{(k+1)}-\bar{p}_l^{(k)}|^2}{\kappa^{(n)}} < \Lambda_1, \quad \quad \forall \ n.
\end{eqnarray}
where $\Lambda_1$ is a constant.
\end{lemma}
\begin{proof}
From Lemma \ref{lem1} the following holds for any $i, i \in \mathcal{A}$,
\begin{eqnarray}
\label{11}
\kappa^{(t+1)}\Upsilon_i(\vect{p}_i^{(t+1)})-\sum_{l \in \mathcal{L}_i}|p_{i,l}^{(t+1)}-p_l|^2+\sum_{l \in \mathcal{L}_i}|\bar{p}_{l}^{(t)}-p_l|^2\nonumber \\
-\sum_{l \in \mathcal{L}_i}|p_{i,l}^{(t+1)}-\bar{p}_l^{(t)}|^2 \geq \kappa^{(t+1)}\Upsilon_i(\vect{p})
\end{eqnarray}
Adding \eqref{11} over all $i$ implies
\begin{eqnarray}
\label{12}
&&\sum_{i\in \mathcal{A}}\kappa^{(t+1)}\Upsilon_i(\vect{p}_i^{(t+1)})-\sum_{i\in \mathcal{A}}\sum_{l \in \mathcal{L}_i}|p_{i,l}^{(t+1)}-p_l|^2\nonumber\\
&&+\sum_{i\in \mathcal{A}}\sum_{l \in \mathcal{L}_i}|\bar{p}_{l}^{(t)}-p_l|^2-\sum_{i\in \mathcal{A}}\sum_{l \in \mathcal{L}_i}|p_{i,l}^{(t+1)}-\bar{p}_l^{(t)}|^2 \nonumber \\
&&\quad \quad\geq \sum_{i\in \mathcal{A}}\kappa^{(t+1)}\Upsilon_i(\vect{p}).
\end{eqnarray}
By a simple algebra we can obtain
\begin{eqnarray}
\label{13}
\sum_{i\in \mathcal{A}}\sum_{l \in \mathcal{L}_i}|\bar{p}_{l}^{(t)}-p_l|^2=\sum_{l \in \mathcal{L}}|\mathcal{A}_l||\bar{p}_{l}^{(t)}-p_l|^2.
\end{eqnarray}
Furthermore, due to the convexity of $|\cdot|^2$, observes that 
\begin{eqnarray}
\label{14}
|\bar{p}_{l}^{(t+1)}-p_l|^2 \leq \frac{\sum_{i \in \mathcal{A}_l}|\bar{p}_{i,l}^{(t+1)}-p_l|^2}{|\mathcal{A}_l|},
\end{eqnarray}
thus,
\begin{eqnarray}
\label{15}
\sum_{l \in \mathcal{L}}|\mathcal{A}_l||\bar{p}_{l}^{(t+1)}-p_l|^2 \leq \sum_{l \in \mathcal{L}}\sum_{i \in \mathcal{A}_l}|\bar{p}_{i,l}^{(t+1)}-p_l|^2.
\end{eqnarray}
Employing \eqref{13} and \eqref{15} in \eqref{12} imply that
\begin{eqnarray}
\label{16}
&&\sum_{i\in \mathcal{A}}\kappa^{(t+1)}\Upsilon_i(\vect{p}_i^{(t+1)})-\sum_{l \in \mathcal{L}}|\mathcal{A}_l||\bar{p}_{l}^{(t+1)}-p_l|^2\nonumber \\
&&+\sum_{l \in \mathcal{L}}|\mathcal{A}_l||\bar{p}_{l}^{(t)}-p_l|^2-\sum_{i\in \mathcal{A}}\sum_{l \in \mathcal{L}_i}|p_{i,l}^{(t+1)}-\bar{p}_l^{(t)}|^2 \nonumber \\
&& \quad \quad \geq \sum_{i\in \mathcal{A}}\kappa^{(t+1)}\Upsilon_i(\vect{p}).
\end{eqnarray}
Adding \eqref{16} over  $t, t=0,1,\cdots, n-1$ results in
\begin{eqnarray} 
\label{17}
&&\sum_{t=0}^{n-1}\sum_{i\in \mathcal{A}}\kappa^{(t+1)}\Upsilon_i(\vect{p}_i^{(t+1)})-\sum_{t=0}^{n-1}\sum_{i\in \mathcal{A}}\sum_{l \in \mathcal{L}_i}|p_{i,l}^{(t+1)}-\bar{p}_l^{(t)}|^2 \nonumber\\
&&-\sum_{l \in \mathcal{L}}|\mathcal{A}_l|\sum_{t=0}^{n-1} \left\{|\bar{p}_{l}^{(t+1)}-p_l|^2 -|\bar{p}_{l}^{(t)}-p_l|^2\right\}\nonumber\\
&& \quad \quad  \geq \sum_{t=0}^{n-1}\sum_{i\in \mathcal{A}}\kappa^{(t+1)}\Upsilon_i(\vect{p}).
\end{eqnarray}
>From \eqref{17} we get 
\begin{eqnarray} 
\label{18}
&&\sum_{t=0}^{n-1}\sum_{i\in \mathcal{A}}\kappa^{(t+1)}\Upsilon_i(\vect{p}_i^{(t+1)})-\sum_{l \in \mathcal{L}}|\mathcal{A}_l||\bar{p}_{l}^{(n)}-p_l|^2\nonumber\\
&& -\sum_{t=0}^{n-1}\sum_{i\in \mathcal{A}}\sum_{l \in \mathcal{L}_i}|p_{i,l}^{(t+1)}-\bar{p}_l^{(t)}|^2 \nonumber \\
&& \quad \quad  \geq \kappa^{(n)}\sum_{i\in \mathcal{A}}\Upsilon_i(\vect{p})-\sum_{l \in \mathcal{L}}|\mathcal{A}_l||\bar{p}_{l}^{(0)}-p_l|^2.
\end{eqnarray}
Since for any $i, i\in \mathcal{A},$ $\Upsilon_i(\vect{p}_i)$ is affine in $\vect{p}_i$, thus 
\begin{eqnarray} 
\label{19}
\frac{\sum_{t=0}^{n-1}\kappa^{(n+1)}\Upsilon_i(\vect{p}_i^{(t+1)})}{\kappa^{(n)}}=\Upsilon_i(\vect{w}_i^{(n)}),
\end{eqnarray}
subsequently,
\begin{eqnarray} 
\label{20}
\sum_{t=0}^{n-1}\sum_{i \in \mathcal{A}}\kappa^{(n+1)}\Upsilon_i(\vect{p}_i^{(t+1)})=\kappa^{(n)}\sum_{i \in \mathcal{A}}\Upsilon_i(\vect{w}_i^{(n)}),
\end{eqnarray}
Substituting \eqref{20} in \eqref{18} and dividing by $\kappa^{(n)}$ result in
\begin{eqnarray} 
\label{21}
\sum_{i \in \mathcal{A}}\Upsilon_i(\vect{w}_i^{(n)})-\frac{\sum_{l \in \mathcal{L}}|\mathcal{A}_l||\bar{p}_{l}^{(n)}-p_l|^2}{\kappa^{(n)}} -\frac{\sum_{t=0}^{n-1}\sum_{i\in \mathcal{A}}\sum_{l \in \mathcal{L}_i}|p_{i,l}^{(t+1)}-\bar{p}_l^{(t)}|^2}{\kappa^{(n)}} \nonumber \\
 \geq \sum_{i\in \mathcal{A}}\Upsilon_i(\vect{p})-\frac{\sum_{l \in \mathcal{L}}|\mathcal{A}_l||\bar{p}_{l}^{(0)}-p_l|^2}{\kappa^{(n)}}.
\end{eqnarray}
Since $\Delta$ is a bounded set and $\lim_{n\rightarrow \infty} \kappa^{(n)}=\infty$, thus 
\begin{eqnarray} 
\label{22}
\lim_{n\rightarrow \infty}\frac{\sum_{l \in \mathcal{L}}|\mathcal{A}_l||\bar{p}_{l}^{(n)}-p_l|^2}{\kappa^{(n)}}=0,\\
\label{23}
\lim_{n\rightarrow \infty}\frac{\sum_{l \in \mathcal{L}}|\mathcal{A}_l||\bar{p}_{l}^{(0)}-p_l|^2}{\kappa^{(n)}}=0.
\end{eqnarray}
Similarly,
\begin{eqnarray}
\label{24}
\lim_{n\rightarrow \infty}\frac{\sum_{t=0}^{n-1}\sum_{i\in \mathcal{A}}\sum_{l \in \mathcal{L}_i}|p_{i,l}^{(t+1)}-\bar{p}_l^{(t)}|^2}{\kappa^{(n)}}=0.
\end{eqnarray}
Consequently, using \eqref{22},\eqref{23} and \eqref{24} in \eqref{21} and the fact that $\Upsilon_i(\vect{p}_i)$ for any $i, i\in \mathcal{A},$ is bounded, the proof of the lemma is complete, i.e., there exists $\Lambda_1$ such that
\begin{eqnarray}
\label{25}
\frac{\sum_{l \in \mathcal{L}_i}\sum_{k=0}^{n-1}|p_{i,l}^{(k+1)}-\bar{p}_l^{(k)}|^2}{\kappa^{(n)}} < \Lambda_1, \quad \quad \forall \ n.
\end{eqnarray}
\end{proof}

\begin{lemma}
\label{lem3}
For any $i, i\in \mathcal{A}$, we have
\begin{eqnarray}
\label{29}
\lim_{n'\rightarrow \infty} |w_{i,l}^{(n')}-w_{l}^{(n')}|^2=0, \quad \forall \ l \in \mathcal{L}_i.
\end{eqnarray}
\end{lemma}

\begin{proof}
Since $|\cdot|^2$ is convex, thus
\begin{eqnarray}
\label{30}
|w_{i,l}^{(n')}-w_{l}^{(n')}|^2 \leq \frac{\sum_{t=0}^{n'-1}\kappa^{(t+1)}|p_{i,l}^{(t+1)}-p_l^{(t)}|^2}{\kappa^{(n)}}.
\end{eqnarray}
Now, we must show that 
\begin{eqnarray}
\label{31}
\forall \ \epsilon>0, \exists n_{\epsilon}, s.t., \forall \ n'>n_{\epsilon}, \sum_{l \in \mathcal{L}_i}|w_{i,l}^{(n')}-w_{l}^{(n')}|^2<\epsilon. 
\end{eqnarray}
>From \eqref{30}, for any $n'>n_{\epsilon}$ we have,
\begin{eqnarray*}
\label{32}
\sum_{l \in \mathcal{L}_i}|w_{i,l}^{(n')}-w_{l}^{(n')}|^2 &\leq & \frac{\sum_{l \in \mathcal{L}_i}\sum_{t=0}^{n'-1}\kappa^{(t+1)}|p_{i,l}^{(t+1)}-p_l^{(t)}|^2}{\kappa^{(n')}}\nonumber \\ 
&=& \frac{\sum_{l \in \mathcal{L}_i}\sum_{t=0}^{n_{\epsilon}-1}\kappa^{(t+1)}|p_{i,l}^{(t+1)}-p_l^{(t)}|^2}{\kappa^{(n')}}
 +\frac{\sum_{l \in \mathcal{L}_i}\sum_{t=n_{\epsilon}}^{n'-1}\kappa^{(t+1)}|p_{i,l}^{(t+1)}-p_l^{(t)}|^2}{\kappa^{(n')}}\nonumber\\
&\leq& \frac{\sum_{l \in \mathcal{L}_i}\kappa^{(1)}\sum_{t=0}^{n_{\epsilon}-1}|p_{i,l}^{(t+1)}-p_l^{(t)}|^2}{\kappa^{(n')}} 
 +\frac{\sum_{l \in \mathcal{L}_i}\kappa^{(n_{\epsilon})}\sum_{t=n_{\epsilon}}^{n'-1}|p_{i,l}^{(t+1)}-p_l^{(t)}|^2}{\kappa^{(n')}}\nonumber \\
&=&\frac{\kappa^{(1)}\sum_{l \in \mathcal{L}_i}\sum_{t=0}^{n_{\epsilon}-1}|p_{i,l}^{(t+1)}-p_l^{(t)}|^2}{\kappa^{(n')}}
 +\frac{\kappa^{(n_{\epsilon})}\sum_{l \in \mathcal{L}_i}\sum_{t=n_{\epsilon}}^{n'-1}|p_{i,l}^{(t+1)}-p_l^{(t)}|^2}{\kappa^{(n')}}.
\end{eqnarray*}
Since $\Delta$ is bounded, thus there exists $\Lambda_2$ such that
\begin{eqnarray}
\label{33}
\sum_{l \in \mathcal{L}_i}\sum_{t=0}^{n_{\epsilon}-1}|p_{i,l}^{(t+1)}-p_l^{(t)}|^2<\Lambda_2.
\end{eqnarray}
Furthermore, $\kappa^{(n')}$ tends to infinity as $n'$ goes to infinity, thus we can choose $n_{\epsilon}^1$ large enough so that
\begin{eqnarray}
\label{34}
\kappa^{(n')}>\frac{2\kappa^{(1)}\Lambda_2}{\epsilon}, \quad \forall \ n'>n_{\epsilon}^1.
\end{eqnarray}
Moreover, using Lemma \ref{lem2} implies that
\begin{eqnarray}
\label{35}
\frac{\sum_{l \in \mathcal{L}_i}\sum_{t=n_{\epsilon}}^{n'-1}|p_{i,l}^{(t+1)}-p_l^{(t)}|^2}{\kappa^{(n')}}<\Lambda_1.
\end{eqnarray} 
Since $\kappa^{(n')}$ tends to zero as $n'$ goes to infinity, thus 
\begin{eqnarray}
\label{36}
\exists \ n_{\epsilon}^2, \quad \  s.t.  \quad \ \kappa^{(n_{\epsilon})}<\frac{\epsilon}{2\Lambda_1}.
\end{eqnarray}
Choosing $n_{\epsilon}=\max \{n_{\epsilon}^1,n_{\epsilon}^2\}$ together with \eqref{34} and \eqref{36} implies that 
\begin{eqnarray}
\label{37}
\sum_{l \in \mathcal{L}_i}|w_{i,l}^{(n')}-w_{l}^{(n')}|^2 <\epsilon, \quad \forall \ n'>n_{\epsilon}, 
\end{eqnarray}
Therefore, 
\begin{eqnarray}
\label{38}
\lim_{n'\rightarrow \infty}\sum_{l \in \mathcal{L}_i}|w_{i,l}^{(n')}-w_{l}^{(n')}|^2
 \quad \quad =\sum_{l\in \mathcal{L}_i}\left\{\lim_{n'\rightarrow \infty}|w_{i,l}^{(n')}-w_{l}^{(n')}|^2\right\}=0.  
\end{eqnarray}
Consequently, $\lim_{n'\rightarrow \infty}|w_{i,l}^{(n')}-w_{l}^{(n')}|^2=0$, and the proof is complete. 
\end{proof}

\begin{corollary}
\label{cor1}
Let for some $i, i\in \mathcal{A}$, $\lim_{n'\rightarrow \infty}w_{i,l}^{(n')}=w_{i,l}^*$, then 
\begin{enumerate}
	\item $\lim_{n'\rightarrow \infty}|w_{j,l}^{(n')}-w_{l}^{(n')}|^2=0$, $\quad \forall \ j \in \mathcal{A}_l, \ l\in L.$
	\item $\lim_{n'\rightarrow \infty}w_{j,l}^{(n')}=w_{i,l}^*=:w_l^*.$ 
\end{enumerate}
\end{corollary}

\begin{proof}
The proof is a direct result of Lemmas \ref{lem0}, \ref{lem2} and \ref{lem3}. 
\end{proof}
Now, using corollary \ref{cor1} in \eqref{21} implies that 
\begin{eqnarray}
\label{38}
\sum_{i\in \mathcal{A}}\Upsilon_i(\vect{w}_i^*)\geq \sum_{i\in \mathcal{A}}\Upsilon_i(\vect{p}), \quad \forall \ \vect{p}\in \Delta.
\end{eqnarray}
Now, since by the assumptions $U_i(\cdot)$ is strictly concave for any $i=1,2,\cdots,m$,  the Karush-Kuhn-Tucker (KKT) conditions are necessary and sufficient to guarantee the optimality of the allocation $\vect{x}^*=(\vect{x}_1^*,\vect{x}_2^*,\cdots,\vect{x}_m^*)$, the limit point of the sequence $\{x_{i,l}^n\}_{n=1}^{\infty}$, that satisfies them.
The Lagrangian for problem  \textbf{Max1} is
\begin{eqnarray}
\mathbb{L}(\vect{x},\vect{\lambda})=\sum_{i=1}^mU_i(\vect{x}_i)-\sum_{l \in \mathcal{L}}\lambda_l\left[\sum_{i: l\in \mathcal{L}_i}x_{i,l}-c_l\right]
\end{eqnarray}

and the KKT conditions are:
\begin{eqnarray}
\label{50}
\frac{\partial \mathbb{L}(\vect{x},\vect{\lambda})}{\partial x_{i,l}}|_{(\vect{x}^*,\vect{\lambda}^*)}=\frac{\partial U_i(\vect{x}^*)}{\partial x_{i,l}}-\lambda_{l}^*&=&0 \quad \quad \forall i, \ l\\
\label{51}
\lambda_{l}^*\left[\sum_{i: l\in \mathcal{L}_i}x_{i,l}^*-c_l\right]&=&0 \quad \quad \forall l
\end{eqnarray} 

Now, if we set $\lambda_l^*=w_l^*$ (notice that, due to Corollary \ref{cor1}, $w_l^*=w_{i,l}^*=w_{-i,l}^*$ for every $i,  l $), then, due to \eqref{38}, we have
\begin{eqnarray}
&&\frac{\partial \Upsilon_i(\vect{p})}{\partial p_{i,l}}=0 \Rightarrow\nonumber \\
&& w_{-i,l}^*\left(\frac{\sum_{j \in \mathcal{A}_l, j \neq i}x_{j,l}^*-c_l}{\gamma}\right)=w_{l}^*\left(\frac{\sum_{j \in \mathcal{A}_l, j \neq i}x_{j,l}^*-c_l}{\gamma}\right)=0, \quad \forall  l\in \mathcal{L},
\end{eqnarray}
which coincides with \eqref{51}. Furthermore, employing Corollary \ref{cor1} along with Lemma \ref{lemma5} imply that 	
\begin{eqnarray}
\frac{\partial\left\{U_i(\vect{x}_i)-\sum_{l\in \mathcal{L}_i}t_{i,l}\right\}}{\partial x_{i,l}}|_{x_{i,l}=x_{i,l}^*,\vect{w}_i=\vect{w}_l^*}=\frac{\partial U_i(\vect{x}_i)}{\partial x_{i,l}}|_{x_{i,l}=x_{i,l}^*}-w_{l}^*=0 \quad \forall \ i, l.
\end{eqnarray}
which coincides with \eqref{50}.
Therefore, the KKT conditions are satisfied and the proof is complete.	

\end{proof}

\end{document}